\documentclass[a4paper]{amsart}

\usepackage{amsmath,amssymb,amsthm}
\usepackage[all]{xy}
\usepackage{eucal,bbm}
\usepackage[dvipdfmx]{graphicx}
\usepackage[dvipdfmx,colorlinks=true,backref=page]{hyperref}

\newtheorem{theorem}{Theorem}[section]
\newtheorem{proposition}[theorem]{Proposition}
\newtheorem{lemma}[theorem]{Lemma}
\newtheorem{corollary}[theorem]{Corollary}

\theoremstyle{definition}
\newtheorem{definition}[theorem]{Definition}
\newtheorem{example}[theorem]{Example}
\newtheorem{conjecture}[theorem]{Conjecture}

\theoremstyle{remark}

\numberwithin{equation}{section}

\newcommand{\Z}{\mathbb{Z}}

\newcommand{\R}{\mathbb{R}}
\newcommand{\C}{\mathbb{C}}

\newcommand{\ind}{\operatorname{ind}}
\newcommand{\Zero}{\operatorname{Zero}}

\SelectTips{cm}{}

\title[Vector fields on non-compact manifolds]{Vector fields on non-compact manifolds}

\author[T. Kato]{Tsuyoshi Kato}
\address{Department of Mathematics, Kyoto University, Kyoto 606-8502, Japan}
\email{tkato@math.kyoto-u.ac.jp}

\author[D. Kishimoto]{Daisuke Kishimoto}
\address{Faculty of Mathematics, Kyushu University, Fukuoka 819-0395, Japan}
\email{kishimoto@math.kyushu-u.ac.jp}

\author[M. Tsutaya]{Mitsunobu Tsutaya}
\address{Faculty of Mathematics, Kyushu University, Fukuoka 819-0395, Japan}
\email{tsutaya@math.kyushu-u.ac.jp}

\date{\today}



\begin{document}

\maketitle

\begin{abstract}
  Let $M$ be a non-compact connected manifold with a cocompact and properly discontinuous action of a discrete group $G$. We establish a Poincar\'{e}-Hopf theorem for a bounded vector field on $M$ satisfying a mild condition on zeros. As an application, we show that such a vector field must have infinitely many zeros whenever $G$ is amenable and the Euler characteristic of $M/G$ is non-zero.
\end{abstract}

%%%%% Section 1 %%%%%

\section{Introduction}\label{Introduction}

%If a manifold $M$ is compact and orientable, then it admits a non-vanishing vector field if and only if the Euler class is trivial. The Poincar\'e-Hopf theorem describes more precise information of zeros of vector fields on $M$, and recovers the above fact.

Let $M$ be a non-compact connected manifold. Then $M$ admits a non-vanishing vector field as $M$ admits a vector field with isolated zeros which can be swept out to infinity. However, the resulting non-vanishing vector field is not satisfactory because it may not be bounded in the following sense. A vector field $v$ on a Riemannian manifold is \emph{bounded} if both $|v|$ and $|dv|$ are bounded, where $dv$ denotes the derivative of $v$. Note that our boundedness condition is different from the one in \cite{CL} and its related works. Bounded vector fields appear in the study of manifolds of bounded geometry. Now we ask whether or not a non-vanishing bounded vector field exists on $M$. Weinberger \cite[Theorem 1]{W} proved that a manifold $M$ of bounded geometry has a non-vanishing vector field $v$ with $|v|$ constant and $|dv|$ bounded if and only if the Euler class of $M$ in the bounded de Rham cohomology $\widehat{H}^*(M)$ is trivial. As one may think of the Poincar\'e-Hopf theorem as a refinement of the Euler class criterion for the existence of a non-vanishing vector field on a compact orientable manifold, it is natural to ask whether or not one can establish the Poincar\'e-Hopf theorem for bounded vector fields on non-compact manifolds of bounded geometry.

A typical manifold of bounded geometry is a covering space of a compact manifold, which we equip with a lift of a metric on the base compact manifold. We say that an action of a group $G$ on a space $X$ is properly discontinuous if every point $x\in X$ has a neighborhood $U$ such that $gU\cap U\ne\emptyset$ implies $g=1$. Equivalently, the quotient map $X\to X/G$ is a covering. So we consider a connected non-compact manifold $M$ on which a cocompact and properly discontinuous action of a group $G$ is given, and will establish the Poincar\'e-Hopf theorem for a bounded vector field on $M$. To state it, we set notation. Let $\ell^\infty(G)$ denote the vector space of bounded functions $G\to\R$, and let $G$ act on $\ell^\infty(G)$ from the left by
\[
  (g\phi)(h)=\phi(hg)
\]
for $g,h\in G$ and $\phi\in\ell^\infty(G)$. We define the module of coinvariants of $\ell^\infty(G)$ by
\[
  \ell^\infty(G)_G=\ell^\infty(G)/\langle\phi-g\phi\mid\phi\in\ell^\infty(G),\,g\in G\rangle
\]
where $\langle S\rangle$ denotes the vector subspace of $\ell^\infty(G)$ generated by a subset $S\subset\ell^\infty(G)$. Let $\mathbbm{1}\in\ell^\infty(G)$ denote the constant function with value $1$. Let $D\subset M$ be a fundamental domain (its definition is given in Section \ref{Fundamental domain}). We will define the index $\mathrm{ind}(v)$ of a bounded vector field $v$ on $M$ as an element of $\ell^\infty(G)_G$, and will prove
\[
  \ind(v)(g)=\sum_{x\in\mathrm{Zero}(v)\cap gD}\mathrm{ind}_x(v)
\]
whenever $v$ is strongly tame, which is a mild condition on the zeros of $v$ defined in Definition \ref{strongly tame}, where $\ind_x(v)$ denotes the local index of a vector field $v$ at the zero $x\in\Zero(v)$. Here, the above equality means that there is a representative $\phi\in\ell^\infty(G)$ of $\ind(v)\in\ell^\infty(G)_G$ such that $\phi(g)$ is the right hand side of the equality. Now we are ready to state the Poincar\'e-Hopf theorem for a bounded vector field on $M$.

\begin{theorem}
  \label{Poincare-Hopf}
  Let $M$ be a non-compact connected manifold equipped with a cocompact and properly discontinuous action of a group $G$ such that $M/G$ is orientable. If a vector field $v$ on $M$ is strongly tame and bounded, then we have
  \[
    \mathrm{ind}(v)=\chi(M/G)\mathbbm{1}\quad\text{in}\quad\ell^\infty(G)_G.
  \]
\end{theorem}

As an application of Theorem \ref{Poincare-Hopf}, we will prove:

\begin{theorem}
  \label{zeros}
  Let $M$ and $G$ be as in Theorem \ref{Poincare-Hopf}. If $G$ is amenable and $\chi(M/G)\ne 0$, then every tame bounded vector field on $M$ must have infinitely many zeros.
\end{theorem}

Let $M$ and $G$ be as in Theorem \ref{Poincare-Hopf}. Then by the above mentioned result of Weinberger \cite[Theorem 1]{W} together with \cite{ABW}, one can deduce that a vector bundle $v$ on $M$ with $|v|$ constant and $|dv|$ bounded must have a zero. But one cannot deduce further information on zeros, such as their numbers, from these results. As an application of Theorem \ref{zeros}, we will get the following result, where tameness of a diffeomorphism is defined in Definition \ref{tame diffeo} in an analogous way to tameness of a vector field.

%As an application of Theorem \ref{zeros}, we will get the following, where tameness of a diffeomorphism is an analogy of tameness of a vector field. We refer the reader to Section \ref{Poincare-Hopf theorem} for the definition of tameness for diffeomorphisms.

\begin{corollary}
  [{cf. \cite[Corollary to Theorem 1]{W}}]
  \label{diffeo}
  Let $M$ and $G$ be as in Theorem \ref{Poincare-Hopf}. If $G$ is amenable and $\chi(M/G)\ne 0$, then every tame diffeomorphism of $M$ which is $C^1$ close to the identity map must have infinitely many fixed points.
\end{corollary}

\begin{example}
  Let $L$ be the non-compact surface called Jacob's ladder, a surface with infinite genus and two ends, which admits an infinite cyclic covering map onto the closed oriented surface of genus $2$. Then we can apply Corollary \ref{diffeo} to $L$, and conclude that any tame diffeomorphism of $L$ which is $C^1$ close to the identity map must have infinitely many zeros. This can be easily generalized to the infinite connected sum $M=\sharp^\infty N$ of a closed connected oriented even dimensional manifold $N$ with $\chi(N)\ne 2$.
\end{example}

We briefly describe the strategy of our proof, as well as some of the tools we exploit. Let $M$ and $G$ be as in Theorem \ref{Poincare-Hopf}. Recall that the Poincar\'e-Hopf theorem for a compact manifold can be proved by using a suitable integral in top dimensional de Rham
cohomology. Motivated by the compact case, we will define the integral
\begin{equation}
  \label{integral}
  \int_M\colon\widehat{H}^n(M)\to\ell^\infty(G)_G
\end{equation}
where $\dim M=n$, and will prove Theorem \ref{Poincare-Hopf} by using it similarly to the compact case. So our approach is an extension of the classical case by means of $\ell^\infty(G)_G$. However, unlikely to the compact case, the target module $\ell^\infty(G)_G$ of the integral has some interesting algebraic properties we will use to deduce Theorem \ref{zeros}.

%, and unlikely to the compact case, the target module $\ell^\infty(G)_G$ of the integral has interesting algebraic properties. We will use some of the algebraic properties of $\ell^\infty(G)_G$ to deduce Theorem \ref{zeros}.

Let us observe possible connections of our results to other contexts. Our results could be connected to the index theory on open manifolds by Roe \cite{R}. More specifically, our index could be related to the index of the Dirac operator
\[
d+d^\ast\colon\widehat{\Omega}^{\mathrm{even}}(M)\to\widehat{\Omega}^{\mathrm{odd}}(M)
\]
on the bounded de Rham complex $\widehat{\Omega}^\ast(M)$, which lives in the operator $K$-theory $K_\ast(C^\ast_u(|G|))$ of the uniform Roe algebra $C^\ast_u(|G|)$. There is another possible connection. In \cite{KKT3}, the pushforward of a vector bundle on $M$ to $M/G$ is defined, and its structure group is the group of unitary operators with finite propagation on the Hilbert space of square integrable functions $G\to\C$. On the other hand, as in \cite{KKT1,KKT2}, the module of coinvariants of bounded functions $\Z\to\Z$ appears in the homotopy groups of such a group of unitary operators of finite propagation for $G=\Z$. Then our results could be connected to the obstruction theory for the pushforward of $TM$ onto $M/G$.

As mentioned in \cite{BW}, there is an isomorphism $\widehat{H}^n(M)\cong H^\mathrm{uf}_0(M)$ (see \cite{AB} for the proof), where $\dim M=n$ and $H^\mathrm{uf}_*(M)$ denotes the uniformly finite homology of $M$ as in \cite{BW}. Since uniformly finite homology is a quasi-isometry invariant, there is an isomorphism $H^\mathrm{uf}_0(M)\cong H^\mathrm{uf}_0(G)$. On the other hand, as in \cite{BNW}, there is an isomorphism $H^\mathrm{uf}_0(G)\cong\ell^\infty(G)_G$. Then we get an isomorphism
\[
  \widehat{H}^n(M)\cong\ell^\infty(G)_G.
\]
However, this isomorphism is not explicit as it is given by a zig-zag of several isomorphisms. We believe that the integral \eqref{integral} gives a direct and explicit description of this isomorphism. Our intuition relies on the case $M = \R$ and $G = \Z$ which we treat in Proposition \ref{M=R}, and we propose the following conjecture:

%We believe that the integral \eqref{integral} gives an explicit description of this isomorphism, and so we pose the following conjecture, where we will verify the conjecture for $M=\R$ and $G=\Z$ in Section \ref{Integral in bounded cohomology}.

\begin{conjecture}
  \label{conjecture}
  The integral \eqref{integral} is an isomorphism.
\end{conjecture}

Throughout this paper, manifolds will be smooth and without boundary, unless otherwise specified, and group actions on manifolds will be smooth too.

\subsection*{Acknowledgement}

The authors were partially supported by JSPS KAKENHI Grant Numbers JP22H01123 (Kato), JP22K03284 and JP19K03473 (Kishimoto), JP22K03317 (Tsutaya). The authors deeply appreciate the referees' useful advice and comments. Especially, Section \ref{Module of coinvariants} was significantly improved by the referees' advice.

%%%%% Section 2 %%%%%

\section{Module of coinvariants}\label{Module of coinvariants}

In this section, we collect properties of the module of coinvariants $\ell^\infty(G)_G$ that we are going to use. Block and Weinberger \cite{BW} introduced the uniformly finite homology $H^\mathrm{uf}_*(X)$ of a metric space $X$, and showed basic properties of it. Later, Brodzki, Niblo, and Wright \cite{BNW} studied amenability of discrete groups by using the uniformly finite homology, where every discrete group will be equipped with a word metric. They observed that if $G$ is finitely generated, then the uniformly finite chain complex $C_*^\mathrm{uf}(G)$ is naturally isomorphic to the chain complex $C_*(G;\ell^\infty(G))$. Then since $H_0(G,\ell^\infty(G))=\ell^\infty(G)_G$, there is a natural isomorphism
\begin{equation}
  \label{uf}
  H_0^\mathrm{uf}(G)\cong\ell^\infty(G)_G
\end{equation}
whenever $G$ is finitely generated.

\begin{proposition}
  \label{quasi-isometry invariant}
  Let $G$ and $H$ be finitely generated groups. Then a quasi-isometric homomorphism $G\to H$ induces an isomorphism
  \[
    \ell^\infty(H)_H\xrightarrow{\cong}\ell^\infty(G)_G.
  \]
\end{proposition}

\begin{proof}
  By \cite[Corollary 2.2]{BW}, a quasi-isometric homomorphism $G\to H$ induces an isomorphism $H^\mathrm{uf}_*(G)\xrightarrow{\cong}H^\mathrm{uf}_*(H)$. Then the statement follows from \eqref{uf}.
\end{proof}

\begin{corollary}
  \label{G finite}
  If $G$ is a finite group, then
  \[
    \ell^\infty(G)_G\cong\R.
  \]
\end{corollary}

\begin{proof}
  Let $1$ denote the trivial group. Since $G$ is finite, the inclusion $1\to G$ is a quasi-isometry. Then since $\ell^\infty(1)_1\cong\R$, the statement is proved by Proposition \ref{quasi-isometry invariant}.
\end{proof}

\begin{proposition}
  \label{finite non-zero}
  Let $G$ be a finitely generated infinite group, and let $\phi\in\ell^\infty(G)$. If $\phi(g)=0$ for all but finitely many $g\in G$, then $\phi$ is zero in $\ell^\infty(G)_G$.
\end{proposition}

\begin{proof}
  In \cite[Theorem 7.6]{Wh}, Whyte gave a necessary and sufficient condition for an element of $C_0^\mathrm{uf}(G)$ to be trivial in $H^\mathrm{uf}_0(G)$. Through the natural isomorphism $C_0^\mathrm{uf}(G)\cong C_0(G;\ell^\infty(G))=\ell^\infty(G)$, this condition is stated as follows. An element $\phi\in\ell^\infty(G)$ is zero in $\ell^\infty(G)_G$ if and only if there are $C>0$ and $r>0$ such that for any finite subset $S\subset G$,
  \[
    \left|\sum_{g\in S}\phi(g)\right|\le C\cdot\sharp\{g\in G\mid 0<d(g,S)\le r\}
  \]
  where $d$ denotes a word metric of $G$. If $G$ is infinite, then for any non-empty finite subset $S\subset G$, we have $\sharp\{g\in G\mid 0<d(g,S)\le 1\}\ge 1$. Suppose $\phi\in\ell^\infty(G)$ satisfies $\phi(g)=0$ for all but finitely many $g\in G$. Then if we set $C=\sum_{g\in G}|\phi(g)|$ and $r=1$, the above inequality holds for any finite subset $S\subset G$, and so $\phi$ is zero in $\ell^\infty(G)_G$.
\end{proof}

Recall that a mean on a group $G$ is a linear map
\[
  \mu\colon\ell^\infty(G)\to\R.
\]
such that $\mu(\mathbbm{1})=1$ and $\mu(\phi)\ge 0$ whenever $\phi(g)\ge 0$ for all $g\in G$, where $\mathbbm{1}\in\ell^\infty(G)$ denotes the constant function with value $1$ as in Section \ref{Introduction}. A group $G$ is \emph{amenable} if it admits a $G$-invariant mean. The proof of \cite[Theorem 3.1]{BW} together with \eqref{uf} implies the following.

\begin{proposition}
  \label{amenable}
  For a finitely generated group $G$, the following statements are equivalent:

  \begin{enumerate}
    \item $G$ is amenable;

    \item $\ell^\infty(G)_G\ne 0$;

    \item $\mathbbm{1}\in\ell^\infty(G)$ is non-zero in $\ell^\infty(G)_G$.
  \end{enumerate}
\end{proposition}

%%%%% Section 3 %%%%%

\section{Basic properties of fundamental domains}\label{Fundamental domain}

In this section, we define a fundamental domain of a manifold with a free group action, and show its basic properties. Throughout this section, let $M$ be a connected manifold of dimension $n$, possibly with boundary, on which a cocompact and properly discontinuous action of a group $G$ is given. Since $G$ is a quotient of the fundamental group of a compact manifold $M/G$, which is finitely generated, $G$ is finitely generated.

We define a \emph{fundamental domain} $D$ of $M$ as the closure of an open set of $M$ such that
\[
  M=\bigcup_{g\in G}gD\quad\text{and}\quad\mathrm{Int}(D)\cap\mathrm{Int}(gD)=\emptyset
\]
for all $1\ne g\in G$. Remark that $D$ needs not be connected. A manifold $M$ admits a fundamental domain. Indeed, given a triangulation of $M/G$, we can lift it to get a triangulation of $M$ such that the $G$-action is free and simplicial. We choose one lift of the interior of each maximal simplex of $M/G$ to $M$, so the closure of the union of these open simplices of $M$ is a fundamental domain of $M$. We choose such a fundamental domain, so that $D$ is a simplicial complex such that each $D\cap gD$ is a subcomplex of $D$ and
\begin{equation}
  \label{boundary D}
  \partial D=\left(\bigcup_{1\ne g\in G}D\cap gD\right)\cup(D\cap\partial M).
\end{equation}
If $gD\cap hD$ is $(n-1)$-dimensional, then we call it a \emph{facet} of $gD$ (and $hD$). We also call $gD\cap\partial M$ a facet of $gD$ when $\partial M\ne\emptyset$. Then the boundary of $D$ is the union of its facets. Clearly, the $G$-action on $M$ restricts to $\partial M$, and $D\cap\partial M$ is a fundamental domain of $\partial M$.

We construct a generating set of $G$ by using a fundamental domain $D$. Let $S$ be a subset of $G$ consisting of elements $g\in G$ such that $D\cap gD$ is a facet of $D$.

\begin{proposition}
  \label{generator}
  The set $S$ is a symmetric generating set of $G$.
\end{proposition}

\begin{proof}
  Let $g\in G$ and $x\in\mathrm{Int}(D)$. Then $gx$ belongs to $\mathrm{Int}(gD)$, and so since $M$ is connected, there is a path $\ell$ from $x$ to $gx$ which passes $g_0D,g_1D,\ldots,g_kD$ in order for $1=g_0,g_1,\ldots,g_{k-1},g_k=g\in G$ such that $g_iD\cap g_{i+1}D$ is a facet and $\ell\cap g_iD\cap g_{i+1}D$ is a single point sitting in the interior of a facet $g_iD\cap g_{i+1}D$ of $g_iD$ for $i=0,1,\ldots,k-1$. Since $g_iD\cap g_{i+1}D=g_i(D\cap g_i^{-1}g_{i+1}D)$ is a facet of $g_iD$, $D\cap g_i^{-1}g_{i+1}D$ is a facet of $D$, implying $g_i^{-1}g_{i+1}\in S$. Thus since
  \[
    g=g_k=(g_0^{-1}g_1)(g_1^{-1}g_2)\cdots(g_{k-1}^{-1}g_k),
  \]
  we obtain that $S$ is a generating set of $G$. If $g\in S$, then $g(D\cap g^{-1}D)=gD\cap D$ is a facet of $D$, and so $D\cap g^{-1}D$ is a facet of $D$ too. Hence $g^{-1}\in S$, that is, $S$ is symmetric, completing the proof.
\end{proof}

\begin{corollary}
  \label{partition}
  There is a partition $S=S^+\sqcup S^-\sqcup S^0$ such that $(S^+)^{-1}=S^-$ and $(S^0)^2=\{1\}$.
\end{corollary}

\begin{proof}
  Let $S^0$ be the subset of $S$ consisting of elements of order two. Then the statement follows because $S$ is symmetric.
\end{proof}

Let $S^+=\{s_1,\ldots,s_k\}$ and $S^0=\{t_1,\ldots,t_l\}$, where $S^+$ and $S^0$ are finite because $G$ is finitely generated as mentioned above. We put
\[
  E=D\cap\partial M,\quad F_i^+=D\cap s_iD,\quad F_i^-=D\cap s_i^{-1}D,\quad F_j^0=D\cap t_jD
\]
for $i=1,2,\ldots,k$ and $j=1,2,\ldots,l$.

\begin{lemma}
  \label{boundary}
  The facets of $D$ are $E,F_1^+,\ldots,F_k^{+},F_1^-,\ldots,F_k^-,F_1^0,\ldots,F_l^0$.
\end{lemma}

\begin{proof}
  The statement follows from Corollary \ref{partition}.
\end{proof}

We consider an orientation of a facet of $gD$.

\begin{lemma}
  \label{orientation}
  Suppose that $M$ is oriented. If $F=gD\cap hD$ is a facet for $g,h\in G$, then the orientations of $F$ induced from $gD$ and $hD$ are opposite.
\end{lemma}

\begin{proof}
  An outward vector of $gD$ rooted at $F$ is an inward vector of $hD$. Then the statement follows.
\end{proof}

%%%%% Section 4 %%%%%

\section{Integral in bounded cohomology}\label{Integral in bounded cohomology}

In this section, we define the integral in bounded cohomology. Let $M$ be a connected Riemannian manifold of dimension $n$, possibly with boundary. As in \cite{R}, we say that a differential form $\omega$ on $M$ is \emph{bounded} if both $|\omega|$ and $|d\omega|$ are bounded. Let $\widehat{\Omega}^p(M)$ denote the set of bounded $p$-forms on $M$. Then by definition, $\widehat{\Omega}^*(M)$ is closed under differential, and so it is a differential graded algebra. We define the \emph{bounded de Rham cohomology} of $M$ as the cohomology of $\widehat{\Omega}^*(M)$, which we denote by $\widehat{H}^*(M)$. We record the following obvious fact.

\begin{lemma}
  \label{induced map}
  If a map $f\colon M\to N$ between manifolds has bounded derivative, then it induces a map $f^*\colon\widehat{\Omega}^*(N)\to\widehat{\Omega}^*(M)$.
\end{lemma}

Now we consider a cocompact and properly discontinuous action of a discrete group $G$ on a manifold $M$, and choose a fundamental domain $D\subset M$. A Riemannian metric of $M$ will be chosen to be the lift of a Riemannian metric of $M/G$. We assume that $M/G$ is oriented. Then in particular, the fundamental domain $D$ is oriented. We define the \emph{integral} of a bounded differential form on $M$ by
\begin{equation}
  \label{integral cochain}
  \int_M\colon\widehat{\Omega}^n(M)\to \ell^\infty(G),\quad\left(\int_M\omega\right)(g)=\int_{gD}\omega.
\end{equation}
We may think of the above integral as the external transfer of the covering $M\to M/G$. Note that we can similarly define the integral for $\partial M$ by using a fundamental domain $D\cap\partial M$ of $\partial M$. We prove Stokes' theorem.

\begin{proposition}
  \label{Stokes}
  For $\omega\in\widehat{\Omega}^{n-1}(M)$, we have
  \[
    \int_Md\omega=\int_{\partial M}\omega\quad\text{in}\quad\ell^\infty(G)_G.
  \]
\end{proposition}

\begin{proof}
  We consider the facets of $D$ described in Lemma \ref{boundary}. Define $\phi_i^\pm,\,\phi^0_j\in\ell^\infty(G)$ by
  \[
    \phi_i^\pm(g)=\int_{gF^\pm_i}\omega\quad\text{and}\quad\phi^0_j(g)=\int_{gF^0_i}\omega
  \]
  for $i=1,2,\ldots,k$ and $j=1,2,\ldots,l$, where the orientations of $gF^\pm_i$ and $gF^0_i$ are induced from $gD$. Then by Lemma \ref{boundary} and the usual Stokes' theorem, we have
  \[
    \int_{gD}d\omega=\int_{gE}\omega+\sum_{i=1}^k(\phi_i^+(g)+\phi_i^-(g))+\sum_{j=1}^l\phi^0_j(g),
  \]
  where the orientation of $gE$ is induced from $gD$. Since $gF_i^-=gs_i^{-1}F_i^+$, it follows from Lemma \ref{orientation} that $\phi_i^-(g)=-\phi_i^+(gs_i^{-1})$. Then we get
  \[
    \phi_i^++\phi_i^-=\phi_i^+-s_i^{-1}\phi_i^+.
  \]
  Quite similarly, we can also get
  \[
    \phi_j^0=\frac{1}{2}(\phi_j^0-t_j^{-1}\phi_j^0).
  \]
  Thus since $E=D\cap\partial M$ is a fundamental domain of $\partial M$, we obtain
  \[
    \int_Md\omega=\int_{\partial M}\omega+\sum_{i=1}^k(\phi_i^+-s_i^{-1}\phi_i^+)+\sum_{j=1}^l\frac{1}{2}(\phi_j^0-t_j^{-1}\phi_j^0).
  \]
  Therefore the proof is finished.
\end{proof}

The following is immediate from Proposition \ref{Stokes}.

\begin{corollary}
  \label{integral cohomology}
  If $M$ is without boundary, then the integral \eqref{integral cochain} induces a map
  \[
    \int_M\colon\widehat{H}^n(M)\to\ell^\infty(G)_G.
  \]
\end{corollary}

By considering $n$-forms with support in $gD$, we can easily see that the integral in bounded cohomology is always surjective. We give two supporting examples for Conjecture \ref{conjecture}.

\begin{proposition}
  Conjecture \ref{conjecture} is true for $G$ finite.
\end{proposition}

\begin{proof}
  If $G$ is finite, then $M$ is compact, and so $\widehat{H}^n(M)$ coincides with the usual $n$-th de Rham cohomology of $M$, which is isomorphic with $\R$. On the other hand, by Corollary \ref{G finite}, we have $\ell^\infty(G)_G\cong\R$. Then since the integral in bounded cohomology is surjective as mentioned above, it is actually an isomorphism, as stated.
\end{proof}

\begin{proposition}
  \label{M=R}
  Conjecture \ref{conjecture} is true for $M=\R$ and $G=\Z$, where $\Z$ acts on $\R$ by translation.
\end{proposition}

\begin{proof}
  We choose the interval $[0,1]\subset\R$ as a fundamental domain. Let $g=1\in\Z$. Suppose that
  \begin{equation}
    \label{kernel}
    \int_\R f(x)dx=\phi-g\phi
  \end{equation}
  for a bounded function $f(x)$ on $\R$ and $\phi\in\ell^\infty(\Z)$, where the integral is taken in the sense of \eqref{integral cochain}. The equation \eqref{kernel} is equivalent to the fact that the $1$-form $f(x)dx$ belongs to the kernel of the integral in bounded cohomology  because
  \[
    \phi-g^n\phi=(\phi+g\phi+\cdots+g^{n-1}\phi)-g(\phi+g\phi+\cdots+g^{n-1}\phi).
  \]
  Note that $(\phi-g\phi)(i)=\phi(i)-\phi(i+1)$. Now we define
  \[
    h(x)=\int_0^xf(t)dt.
  \]
  To see that the integral in bounded cohomology is injective, it is sufficient to show that $h(x)\in\widehat{\Omega}^0(\R)$. Since $dh(x)=f(x)dx$, $dh(x)$ is bounded. For $0\le n\le x<n+1$, we have
  \[
    h(x)=\sum_{i=0}^{n-1}\int_i^{i+1}f(t)dt+\int_n^xf(t)dt=\phi(0)-\phi(n)+\int_n^xf(t)dt.
  \]
  Since $f(x)$ is bounded, $\int_n^xf(t)dt$ is bounded too as $x$ and $n$ vary. Then $h(x)$ is bounded for $x\ge 0$. Quite similarly, we can show that $h(x)$ is bounded for $x<0$ too, and so we get $h(x)\in\widehat{\Omega}^0(\R)$. Thus we obtain the injectivity. Since the integral in bounded cohomology is surjective as mentioned above, it is an isomorphism.
\end{proof}

%We can slightly generalize Proposition \ref{Stokes} as follows. Let $\{U_i\}_{i\in I}$ be a family of open sets of $M$ such that each $U_i\subset\mathrm{Int}(D_g)$ for some $g\in G$ and
%\[
%  \sup_{g\in G}|\{i\in I\mid U_i\subset D_g\}|<\infty.
%\]
%Then $X=M-\bigcup_{i\in I}U_i$ is an $n$-dimensional manifold with boundary. Note that the $G$-action needs not restrict to $\partial X$ because $U$ needs not be $G$-invariant. Now we define the integral by
%\[
%  \int_X\colon\widehat{\Omega}^n(X)\to\ell^\infty(G)_G,\quad\omega\mapsto\left(\int_{X\cap D_g}\omega\right)_{g\in G}.
%\]
%and
%\[
%  \int_{\partial X}\colon\widehat{\Omega}^{n-1}(X)\to\ell^\infty(G)_G,\quad\omega\mapsto\left(\int_{\partial X\cap D_g}\omega\right)_{g\in G}.
%\]
%The following proposition can be proved quite similarly to Proposition \ref{Stokes}, and so we omit the proof.

%\begin{proposition}
%  \label{Stokes 2}
%  For $\omega\in\widehat{\Omega}^{n-1}(X)$, we have
%  \[
%    \int_Xd\omega=\int_{\partial X}\omega.
%  \]
%\end{proposition}

%%%%% Section 5 %%%%%

\section{Poincar\'e-Hopf theorem} \label{Poincare-Hopf theorem}

In this section, we prove Theorems \ref{Poincare-Hopf} and \ref{zeros}. Throughout this section, let $M$ be a connected manifold of dimension $n$ equipped with a cocompact and properly discontinuous action of a discrete group $G$ such that $M/G$ is oriented. The metric of $M$ will be the lift of a metric of $M/G$.

Let $\Phi$ denote a representative of the Thom class of $M/G$. Then as in \cite{BT}, the support of $\Phi$ is compactly supported, and so $\Phi$ is a bounded $n$-form on $T(M/G)$. Let $\pi\colon M\to M/G$ denote the projection. Then the derivative of $\pi$ is bounded, and so by Lemma \ref{induced map}, we get the induced map $\pi^*\colon\widehat{\Omega}^*(T(M/G))\to\widehat{\Omega}^*(TM)$. In particular, $\pi^*(\Phi)$ is a bounded $n$-form on $TM$. Note that $\pi^*(\Phi)$ represents the Thom class of $M$ in bounded cohomology. Let $v$ be a vector field on $M$ with $|dv|$ bounded. Then by Lemma \ref{induced map}, $v^*(\pi^*(\Phi))$ is a bounded $n$-form on $M$, and so we can define the index of $v$ by
\[
  \ind(v)=\int_Mv^*(\pi^*(\Phi))\in\ell^\infty(G)_G.
\]
We remark that the index $\ind(v)$ is independent of the choice of a representative $\Phi$ of the Thom class of $M/G$. Indeed, if $\Psi$ is another representative of the Thom class of $M/G$, then $\Phi-\Psi=d\alpha$ for some compactly supported $(n-1)$-form $\alpha$ on $T(M/G)$, where $\Psi$ is compactly supported. Hence we get $\pi^*(\Phi)-\pi^*(\Psi)=d\pi^*(\alpha)$, where all differential forms are bounded, and so by Corollary \ref{integral cohomology}, the indices of $v$ defined by $\Phi$ and $\Psi$ are equal. We also remark that by Proposition \ref{amenable}, the index of a bounded vector field on $M$ is always zero whenever $G$ is not amenable. (cf. \cite[Theorem 2]{W})

We now show some properties of the index. Let $v_0$ denote the zero vector field, that is, the zero section $M\to TM$. Then $v_0^*(\pi^*(\Phi))$ is a representative of the Euler class $e(M)$ in bounded cohomology, which was considered by Weinberger \cite{W}.

\begin{proposition}
  \label{Euler}
  There is an equality
  \[
    \int_Me(M)=\chi(M/G)\mathbbm{1}.
  \]
\end{proposition}

\begin{proof}
  Let $\bar{v}_0$ denote the zero vector field on $M/G$, so $v_0$ is the lift of $\bar{v}_0$. Since the projection $\pi\colon\mathrm{Int}(gD)\to M/G-\pi(\partial(gD))$ is a diffeomorphism and both $\partial(gD)$ and $\pi(\partial(gD))$ have measure zero, we have
  \[
    \int_{gD}v_0^*(\pi^*(\Phi))=\int_{M/G}\bar{v}_0^*(\Phi)=\int_{M/G}e(M/G)=\chi(M/G).
  \]
  Thus the proof is finished.
\end{proof}

\begin{lemma}
  \label{homotopy}
  If vector fields $v$ and $w$ on $M$ with $|dv|$ and $|dw|$ bounded are homotopic by a homotopy with bounded derivative, then
  \[
    \ind(v)=\ind(w).
  \]
\end{lemma}

\begin{proof}
  Let $v_t\colon M\times[0,1]\to TM$ be a homotopy with bounded derivative such that $v_0=v$ and $v_1=w$. Since the induced maps $v_t^*\colon\widehat{\Omega}^*(TM)\to\widehat{\Omega}^*(M\times[0,1])$ and $\pi^*\colon\widehat{\Omega}^*(T(M/G))\to\widehat{\Omega}^*(TM)$ commute with the differential, we have
  \[
    \int_{M\times[0,1]}dv_t^*(\pi^*(\Phi))=\int_{M\times[0,1]}v_t^*(\pi^*(d\Phi))=0
  \]
  where $\Phi$ is a closed $n$-form representing the Thom class of $T(M/G)$. On the other hand, by Proposition \ref{Stokes}, we have
  \[
    \int_{M\times[0,1]}dv_t^*(\pi^*(\Phi))=\int_{M\times 1}w^*(\pi^*(\Phi))-\int_{M\times 0}v^*(\pi^*(\Phi)).
  \]
  Thus the statement is proved.
\end{proof}

\begin{proposition}
  \label{index}
  Let $v$ be a bounded vector field on $M$. Then we have
  \[
    \ind(v)=\ind(v_0).
  \]
\end{proposition}

\begin{proof}
  Clearly, $tv$ is a homotopy from $v_0$ to $v$ with bounded derivative. Then by Lemma \ref{homotopy}, the proof is done.
\end{proof}

We consider a mild condition on zeros of a vector field. Let $B_\delta(x)$ denote the open $\delta$-neighborhood of $x\in M$, and let $B_\epsilon(M)$ denote the open $\epsilon$-neighborhood of $M$ in $TM$.

\begin{definition}
  A vector field $v$ on a manifold $M$ is \emph{tame} if there are $\delta>0$ and $\epsilon>0$ such that
  \begin{enumerate}
    \item $B_\delta(x)\cap B_\delta(y)=\emptyset$ for $x\ne y\in\mathrm{Zero}(v)$, and
    \item $v^{-1}(B_\epsilon(M))\subset\bigcup_{x\in\mathrm{Zero}(v)}B_\delta(x)$.
  \end{enumerate}
\end{definition}

\begin{definition}
  \label{strongly tame}
  A vector field $v$ on $M$ is \emph{strongly tame} if it is tame and there is $\delta>0$ such that for each $x\in\mathrm{Zero}(v)$, $B_\delta(x)\subset gD$ for some $g\in G$.
\end{definition}

We also define a tame diffeomorphism, in analogy with tame vector fields.

\begin{definition}
  \label{tame diffeo}
  A diffeomorphism $f\colon M\to M$ is \emph{tame} if there are $\delta>0$ and $\epsilon>0$ such that
  \begin{enumerate}
    \item $B_\delta(x)\cap B_\delta(y)=\emptyset$ for $x\ne y\in\mathrm{Fix}(f)$, and
    \item $d(x,f(x))>\epsilon$ for $x\in M-\bigcup_{y\in\mathrm{Fix}(y)}B_\delta(y)$
  \end{enumerate}
  where $d$ stands for the metric of $M$.
\end{definition}

We prove a technical lemma.

\begin{lemma}
  \label{local index}
  Let $f\colon\R^n\to T\R^n$ be a section of the tangent bundle $T\R^n=\R^n\times\R^n\to\R^n$ such that for some $\delta,\epsilon>0$, $f^{-1}(B_\epsilon(0))\subset B_\delta(0)$. Let $\omega$ be a representative of the Thom class of $T\R^n$ such that $\mathrm{supp}(\omega)\subset\R^n\times B_{\epsilon/2}(0)$. Then we have
  \[
    \mathrm{ind}_0(f)=\int_{B_\delta(0)}f^*(\omega).
  \]
\end{lemma}

\begin{proof}
  Let $B$ denote the closure of $B_\delta(0)$. Then by definition, $\mathrm{ind}_0(f)$ is the mapping degree of the composite
  \[
    \partial B\xrightarrow{f}\R^n\times(\R^n\setminus\{0\})\xrightarrow{p}\R^n\setminus\{0\}\xrightarrow{q}S^{n-1}
  \]
  where $p\colon T\R^n=\R^n\times\R^n\to\R^n$ denotes the second projection and $q\colon\R^n\setminus\{0\}\to S^{n-1}$ denotes the natural projection onto the unit sphere. There is a function $\rho\colon[0,\infty)\to\R$ such that $\rho(x)=0$ for $x$ sufficiently close to $0$ and $\rho(x)=1$ for $x\ge\epsilon/2$. Let $\alpha$ be an $(n-1)$-form on the unit sphere $S^{n-1}$ of $\R^n$ such that $\int_{S^{n-1}}\alpha=1$. Now we define
  \[
    \eta=p^*(d\rho\wedge q^*(\alpha)).
  \]
  By definition, we have $\mathrm{supp}(\eta)\subset\R^n\times B_{\epsilon/2}(0)$, and since $\int_{\R^n}d\rho\wedge q^*(\alpha)=1$, $\eta$ represents the Thom class of $T\R^n$. Then by the uniqueness of the Thom class, there is an $(n-1)$-form $\tau$ on $T\R^n$ with vertically compact support such that
  \[
    \omega-\eta=d\tau.
  \]
  We have $\mathrm{supp}(\tau)\subset\R^n\times B_{\epsilon/2}(0)$, implying $f^*(\tau)\vert_{\partial B}=0$. So by Stokes' theorem, we get
  \[
    \int_Bf^*(\omega)-\int_Bf^*(\eta)=\int_Bf^*(d\tau)=\int_Bdf^*(\tau)=\int_{\partial B}f^*(\tau)=0.
  \]
  Note that $\eta=dp^*(\rho\cdot q^*(\alpha))$ and $\rho(f(\partial B))=1$. Then by Stokes' theorem, we have
  \[
    \int_Bf^*(\eta)=\int_Bdf^*\circ p^*(\rho\cdot q^*(\alpha))=\int_{\partial B}f^*\circ p^*(\rho\cdot q^*(\alpha))=\int_{\partial B}f^*\circ p^*\circ q^*(\alpha)=\mathrm{ind}_0(f).
  \]
  Thus since $\int_{B_\epsilon(0)}f^*(\omega)=\int_Bf^*(\omega)$, the proof is finished.
\end{proof}

%By using an isotopy of $M$, we can easily see that if a vector field on $M$ has finitely many zeros, then it is homotopic to a strongly tame vector field by a homotopy with bounded derivative.

We compute the index of a strongly tame bounded vector field.

\begin{proposition}
  \label{local-to-global}
  Let $v$ be a strongly tame bounded vector field on $M$. Then
  \[
    \ind(v)(g)=\sum_{x\in\Zero(v)\cap gD}\ind_x(v).
  \]
\end{proposition}

\begin{proof}
  Let $\delta$ and $\epsilon$ be as in the definition of a tame vector field. As in \cite{BT}, we may assume that the support of $\pi^*(\Phi)$ is contained in $B_{\epsilon/2}(M)$. Then we have
  \[
    \ind(v)(g)=\sum_{x\in\Zero(v)\cap gD}\int_{B_\delta(x)}v^*(\pi^*(\Phi))
  \]
  for each $g\in G$. On the other hand, by tameness of $v$, $\pi^*(\Phi)\vert_{B_\delta(x)}$ is compactly supported for $x\in\mathrm{Zero}(v)$, and so by Lemma \ref{local index}, we get
  \[
    \int_{B_\delta(x)}v^*(\pi^*(\Phi))=\ind_x(v)
  \]
  for each $x\in\Zero(v)$. Thus the statement is proved.
\end{proof}

Now we are ready to prove Theorem \ref{Poincare-Hopf}.

\begin{proof}
  [Proof of Theorem \ref{Poincare-Hopf}]
  Combine Propositions \ref{Euler}, \ref{index} and \ref{local-to-global}.
\end{proof}

\begin{proof}
  [Proof of Theorem \ref{zeros}]
  As mentioned at the beginning of Section \ref{Fundamental domain}, $G$ is finitely generated. Then we can apply the results in Section \ref{Module of coinvariants}. Let $v$ be a tame bounded vector field on $M$, and suppose that $v$ has finitely many zeros. We can easily see that $v$ is homotopic to a strongly tame vector field by a homotopy with bounded derivative. Then by Lemma \ref{homotopy}, we may assume that $v$ itself is strongly tame, and so by Theorem \ref{Poincare-Hopf}, we have
  \[
    \ind(v)=\chi(M/G)\mathbbm{1}.
  \]
  Since $\chi(M/G)\ne 0$, it follows from Proposition \ref{amenable} that $\chi(M/G)\mathbbm{1}$ is non-zero in $\ell^\infty(G)_G$. Then by Proposition \ref{finite non-zero}, we obtain that $v$ must have infinitely many zeros, a contradiction. Thus $v$ must have infinitely many zeros.
\end{proof}

\begin{proof}
  [Proof of Corollary \ref{diffeo}]
  Observe that a tame diffeomorphism is the composite of a tame vector field and the exponential map if it is $C^1$ close to the identity map. Then the statement follows from Theorem \ref{zeros}.
\end{proof}


\begin{thebibliography}{99}
  \bibitem{AB} O. Attie and J. Block, Poincar\'e duality for $L^p$ cohomology and characteristic classes, unpublished.

  \bibitem{ABW} O. Attie, J. Block, and S. Weinberger, Characteristic classes and distortion of diffeomorphisms, J. Amer. Math. Soc. \textbf{5} (1992), 919-921.

  \bibitem{BW} J. Block and S. Weinberger, Aperiodic tilings, positive scalar curvature, and amenability of spaces, J. Amer. Math. Soc. \textbf{5} (1992), no. 4, 907-918.

  \bibitem{BT} R. Bott and L.W. Tu, Differential Forms in Algebraic Topology, Graduate Texts in Mathematics \textbf{82}, Springer-Verlag, New York-Berlin, 1982.

  \bibitem{BNW} J. Brodzki, G.A. Niblo, N. Wright, Pairings, duality, amenability and bounded cohomology, J. Eur. Math. Soc. \textbf{14} (2012), 1513-1518.

  %\bibitem{CG} J. Cheeger and M. Gromov, $L_2$-cohomology and group cohomology, Topology \textbf{25} (1986), 189-215.

  %\bibitem{CGT} J. Cheeger, M. Gromov, and M. Taylor, Finite propagation speed, kernel estimates for functions of the Laplace operator, and the geometry of complete Riemannian manifolds, J. Diff. Geom. \textbf{17} (1982), 15-54.

  %\bibitem{C} P.R. Chernoff, Essential self-adjointness of powers of generators of hyperbolic equations, J. Funct. Anal. \textbf{12} (1973), 401-414.

  \bibitem{CL} A. Cima and J. Llibre, Bounded polynomial vector fields, Trans. Amer. Math. Soc. \textbf{318} (1990), no. 2, 557-579.

  %\bibitem{D} M.J. Dupr\'{e}, Classifying Hilbert bundles, J. Funct. Anal. \textbf{15} (1974), 244-278.

  %\bibitem{E} B. Eckmann, Amenable groups and Euler characteristic, Comment. Math. Helv. \textbf{67} (1992), 383-393.

  %\bibitem{FM} W. Fulton and R. MacPherson, Characteristic classes of direct image bundles for covering maps, Ann. Math. \textbf{125} (1987), 1-92.

  %\bibitem{GNVW} D. Gross, V. Nesme, H. Vogts, and R.F. Werner, Index theory of one Dimensional quantum walks and cellular automata, Commun. Math. Phys. \textbf{310} (2012), 419-454.

  %\bibitem{HR} B. Hughes and A. Ranicki, Ends of Complexes, Cambridge Tracts in Mathematics \textbf{123}, Cambridge, 1996.

  \bibitem{KKT1} T. Kato, D. Kishimoto and M. Tsutaya, Homotopy type of the space of finite propagation unitary operators on $\mathbb{Z}$, Homology Homotopy Appl. \textbf{25} (2023), no. 1, 375-400.

  \bibitem{KKT2} T. Kato, D. Kishimoto, and M. Tsutaya, Homotopy type of the unitary group of the uniform Roe algebra on $\Z^n$, J. Topol. Anal. \textbf{15} (2023), no. 2, 495-512.

  \bibitem{KKT3} T. Kato, D. Kishimoto, and M. Tsutaya, Hilbert bundles with ends, accepted by J. Topol. Anal.

  %\bibitem{K} N. Kuiper, The homotopy type of the unitary group of Hilbert space, Topology \textbf{3}, no. 1, (1965), 19-30.

  %\bibitem{Mil} J.W. Milnor, Topology from the Differential Viewpoint, Princeton Landmarks in Mathematics, 1997.

  %\bibitem{Mis} A. Mischenko, Infinite dimensional representations of discrete groups and higher signatures, Izv. Akad. Nauk SSSR Ser. Mat. \textbf{38} (1974), 81-106.

  %\bibitem{PS} A. Pressley and G. Segal, Loop Groups, Oxford Mathematical Monographs, The Clarendon Press, Oxford University Press, New York, 1986.

  %\bibitem{Q} D. Quillen, The spectrum of an equivariant cohomology ring. I, II, Ann. of Math. (2), \textbf{94} (1971), 549–572; ibid. 573–602.

  \bibitem{R} J. Roe, An index theorem on open manifolds. I, J. Diff. Geom. \textbf{27} (1988), 87-113.

  %\bibitem{R} J. Roe, Elliptic Operators, Topology and Asymptotic Methods, Second Edition, Addison Wesley Longman (1998).

  %\bibitem{S} R.L.E. Schwarzenberger, The direct image of a vector bundle, Math. Proc. Cambridge Phil. Soc. \textbf{63} (1967), no. 2 315-334.

  %\bibitem{Wa} S. Wagon, The Banach--Tarski Paradox, Cambridge University Press, 1985.

  \bibitem{W} S. Weinberger, Fixed point theories on noncompact manifolds, J. Fixed Point Theory Appl. \textbf{6} (2009), no. 1, 15-25.

  \bibitem{Wh}  K. Whyte, Amenability, bilipschitz equivalence, and the von Neumann conjecture, Duke Math. J. \textbf{99} (1999), no. 1, 93-112.
\end{thebibliography}
\end{document}